\documentclass[11pt]{amsart}
\usepackage[utf8]{inputenc}
\usepackage{amsmath, amssymb, amsthm}
\usepackage{enumitem}
\usepackage{geometry}
\usepackage{setspace}
\usepackage{hyperref}

\usepackage{amsfonts,latexsym,amsthm,amssymb,graphicx}
\usepackage[all]{xy}
\usepackage{hyperref}
\usepackage[usenames]{color} 

\DeclareFontFamily{OT1}{rsfs}{}
\DeclareFontShape{OT1}{rsfs}{n}{it}{<-> rsfs10}{}
\DeclareMathAlphabet{\mathscr}{OT1}{rsfs}{n}{it}

\CompileMatrices

\newtheorem{theorem}{Theorem}[section]
\newtheorem{lemma}[theorem]{Lemma}

\newtheorem{conj}{Conjecture}
\newtheorem{prop}[theorem]{Proposition}
\newtheorem{corollary}[theorem]{Corollary}

{\theoremstyle{defin} \newtheorem{defin}[theorem]{Definition}}
{\theoremstyle{remark} \newtheorem{remark}[theorem]{Remark}
}

\title{Regular Functions on Formal-Analytic Arithmetic Surfaces}
\author{Samuel Goodman}
\date{May 2025}

\begin{document}
\maketitle \vspace{-2.5 em} \begin{abstract}
In this paper, we show that for a broad class of pseudoconvex formal-analytic arithmetic surfaces, those which admit a nonconstant monic such regular function, that a conjecture of Bost-Charles that the ring of regular functions has continuum cardinality is implied by a purely complex-analytic conjecture. Under the conjecture, a Fekete–Szëgo-type approximation argument produces a polynomial ``large” relative to the regular function, which in turn yields continuum many distinct regular functions. We also introduce a formula for the pushforward by a holomorphic function of the equilibrium Green's functions for our bordered Riemann surface with boundary, a formula which has constant term related to Arakelov degree. 
\end{abstract}
\maketitle
\section{Introduction} Intersection theory was one of the crucial tools enabled by the algebraic geometry revolution. It proved useful to the point where efforts were made to find a good intersection theory in settings outside the useful framework of intersection theory within algebraic geometry. One of these instances is Arakelov intersection theory, which is defined using the theory of Green's functions on compact Riemann surfaces (with boundary) to get a handle on the infinite (archimedean) places. Since the construction of an Arakelov surface can be thought of as combining the fibers over the finite primes of $\text{Spec}(\mathcal{O}_K)$ with Riemann surfaces for the infinite places (which are in some sense the ``thicker" part of the construction), Green's functions become a natural way to attach numerical invariants to the Riemann surfaces that appear at the infinite places. For the purposes of functoriality, some extra conditions on put on the Green's functions of interest, yielding an Arakelov intersection theory that is well-behaved under pushforwards and pullbacks (so that they induce maps on the Arakelov-Chow groups). \newline \newline We now explain the construction of formal-analytic arithmetic surfaces over $\mathcal{O}_K$, following section $6.1.1$ of [2]. A formal-analytic arithmetic surface $\tilde{\mathcal{V}}$ comprises of the data $(\widehat{\mathcal{V}},(V_{\sigma},O_{\sigma},\iota_{\sigma})_{\sigma: K \to \mathbb{C}})$, where \newline  \newline 1) $\widehat{\mathcal{V}}$ is a Noetherian affine formal scheme over $\mathcal{O}_K$ (this is the part of the surface lying over the finite places) such that $\pi: \tilde{\mathcal{V}}=\text{Spf}(B) \to \text{Spec}(\mathcal{O}_K)$ induces an ismorphism $\pi_{|\mathcal{V}|}: |\mathcal{V}| \to \text{Spec}(\mathcal{O}_K)$ on the reduced schemes of definition and such that $B$ is formally smooth over $\mathcal{O}_K$ with $\pi$ having $1$-dimensional fibers, \newline \newline 2) Each $(V_{\sigma},O_{\sigma},\iota_{\sigma})$ (corresponding to the embeddings $\sigma: K \to \mathbb{C}$) consists of a connected compact Riemann surface (with boundary) $V_{\sigma}$, a distinguished point $O_{\sigma} \in \mathring{V}$, and an isomorphism $\iota_{\sigma}: \widehat{\mathcal{V}}_{\sigma} \to \widehat{V}_{\sigma,O_{\sigma}}$, where $\widehat{\mathcal{V}}_{\sigma}$ is the smooth formal curve $\widehat{\mathcal{V}} \otimes_{\mathcal{O}_K,\sigma} \mathbb{C}$ induced by the topological $\mathbb{C}$-algebras $B \widehat{\otimes}_{\mathcal{O}_K,\sigma} \mathbb{C}$, while $\widehat{V}_{\sigma,O_{\sigma}}$ is the formal completion of $V_{\sigma}$ at $O_{\sigma}$. We also have a compatibility relationship which relates pairs of complex conjugate embeddings (so that there are no redundancies for the complex infinite places). \newline \newline The isomorphism $\pi_{|\mathcal{V}|}: |\mathcal{V}| \to \text{Spec}(\mathcal{O}_K)$ induces a section $P: \text{Spec}(\mathcal{O}_K) \to \widehat{\mathcal{V}}$. Letting $I$ be a maximal ideal of definition for $B$, we construct the corresponding ideal sheaf $\mathcal{I}$ and then define $N_P(\widehat{\mathcal{V}})=P^*(\widehat{\mathcal{I}/\mathcal{I}^2})$, which is the pullback of the tangent bundle $\widehat{\mathcal{I}/\mathcal{I}^2}$ over $\widehat{\mathcal{V}}$  with respect to this section. We can then construct a normal bundle $N_P(\tilde{\mathcal{V}})$ by starting with $N_P(\widehat{\mathcal{V}})$ and then equipping it with capacitary metrics $||\cdot||_{V_{\sigma},P_{\sigma}}$ along the complex lines corresponding to copies of $T_{P_{\sigma}}V_{\sigma}$. \newline \newline Letting $\widehat{\text{deg}}$ be Arakelov degree, the key invariant of interest will be $\widehat{\text{deg}}(\overline{N}_P(\tilde{\mathcal{V}}))$. We say that $\tilde{V}$ is pseudoconvex if $\widehat{\text{deg}}(\overline{N}_P(\tilde{V}))<0$. Our main goal is to prove the following claim: \begin{theorem} Let $\tilde{V}$ be a pseudoconvex formal-analytic arithmetic surface over $\text{Spec}(\mathbb{Z})$, centered at $O$. Let $\varphi$ be the compositional inverse of the gluing map $\psi$. Suppose that $\tilde{V}$ admits a nonconstant monic $f \in \mathbb{Z}[[T]]$ such that $\phi=f \circ \varphi$ is holomorphic on some neighborhood of $V$ and that Conjecture $1$ holds for $\phi$. Then there are a continuum of power series $g \in \mathbb{Z}[[T]]$ such that $g \circ \varphi$ is holomorphic on some neighborhood of $V$. \end{theorem} \noindent
Implicit in the statement of Theorem $1.1$ is the following conjecture, which is purely complex analytic: \begin{conj} Let $\phi: V^+ \to \mathbb{C}$ be holomorphic on a neighborhood of $V$ of vanishing order $e$ at $O$ and let $\sigma$ be an antiholomorphic involution of $V^+$. Suppose that $\log||\phi^{[e]}(O)||^{\text{cap}, \otimes e}_{V,O}<0$ and such that $\phi \circ \sigma=\overline{\phi}$ on $V^+$. Then there exists a monic polynomial $p(X) \in \mathbb{R}[X]$ such that $\inf_{x \in V} |p(\frac{1}{\phi(x)})|>1$. \end{conj} \noindent We will also show that any $\phi$ in the statement of Theorem $1.1$ satisfies the hypotheses of Conjecture $1$. \newline \newline \noindent In the sense of Bost-Charles, see [2], we see that Conjecture $1$ resolves the Bost-Charles Conjecture under the assumption that $\tilde{\mathcal{V}}$ admits a single nonconstant monic regular function. For example, for any $\tilde{\mathcal{V}}$ such that $\varphi$ extends holomorphically on some neighborhood of $V$, the hypothesis is satisfied, resolving the conjecture for all such formal-analytic arithmetic surfaces. \newline \newline Our second main result is a pushforward formula for the equilibrium Green's function of $V$ at $O$ of independent interest that could prove quite useful in proving Conjecture $1$ (see Corollary $3.6$): \begin{theorem} Given a holomorphic $f: V^+ \to \mathbb{C}$ with $f(O)=0$, we have the following identity on $V^+ \backslash \{O\}$: $$(f_*g)(f(x))-(\int_V g\delta_{f^*(f(O))-eO}+\log||f^{[e]}(O)||_{V,O}^{\text{cap},\otimes e})=\int_{V} \log|\frac{1}{f(x)}-\frac{1}{f(x')}|\mu(x')$$ \end{theorem} \noindent This is a generalization of the classical identity for the Green's function for holomorphic pushforwards. One naive attempt to prove Conjecture $1$ using Theorem $1.2$ on $\phi$ first changes the Riemann surface by removing balls around the other zeros (besides $O$) of $\phi$ so that $0$ only has one preimage on $W$ without changing the image of $\phi$. A monotonicity property of capacitary norms then reduces the claim to one on $W$, but an approximation argument for the equilibrium measure, giving Fekete polynomials, fails to prove Conjecture $1$ since the resulting polynomials do not satisfy the condition on all of $V$ (but rather only away from the Fekete points). \newline \newline \noindent We now briefly outline the paper. In Section $2$, we define Riemann surfaces in our context, define the notion of Green's functions and equilibrium Green's functions, the notion of pushforward, and one of capacitary norms. In Section $3$, we build on the work of Bost-Charles to prove  a new formula for the pushforwards of Green's function that relates to Arakelov degree, generalizing the classical formula for the Green's function. We also prove that $\phi$ satisfies both hypotheses of Conjecture $1$ by relating the constant term appearing in the pushforward formula with $\phi$ to Arakelov degree and analyzing the antiholomorphic involution on $V^+$. Finally, in Section $4$, we use some number theoretic approximation lemmas to prove Theorem $1.1$. 
\section{Riemann Surfaces a la Bost-Charles}
First we define Riemann surfaces in our context. We will follow the exposition of section $4.3.1.2$ of [2] throughout. \begin{defin} A Riemann surface with boundary is a pair $(V^+,V)$, where $V^+$ is a (connected) Riemann surface without boundary that is a germ of a Riemann surface $V$ along a $\mathcal{C}^{\infty}$-submanifold with boundary $V$ of $V^+$ (of codimension $0$). Furthermore, we set $\partial V=V^+ \backslash V^{\circ}$, where $V^{\circ}$ is the usual interior of $V$. We also assume that $\partial V$ is nonempty. \end{defin} \begin{remark} This can be thought of as a Riemann surface with boundary $V$ with a slight thickening beyond its usual boundary, giving a small open neighborhood of the manifold $V$, thus smoothly embedding $V$ in $V^+$. \end{remark} Now we can recall the definition of the equilibrium Green's function on a Riemann surface, suitably modified for our case. \begin{defin} Let $(V^+, V)$ be a Riemann surface and $O$ a point of $V^{\circ}$. Then the equilibrium Green's function $g_{V,O}$ of $P$ in $V$ is the unique function $g_{V,O}: V^+ \backslash \{O\} \to \mathbb{R}$ satisfying the following conditions: \newline \newline
1) $g_{V,O}$ is continuous on $V^+ \backslash \{O\}$ and vanishes on $V^+ \backslash V^{\circ}$ \newline \newline
2) $g_{V,O}$ is harmonic on $V^{\circ} \backslash \{O\}$ \newline \newline 
3) $g_{V,O}$ has a logarithmic singularity at $O$, namely if $z$ is a local coordinate of $V$ at $O$, then $g_{V,O}-\log(|z|^{-1})$ is bounded on some pointed neighborhood of $O$. \end{defin} \begin{remark} It immediately follows from the maximum principle for harmonic functions that $g_{V,O}$ is strictly positive on $V^{\circ} \backslash \{O\}$. Furthermore, note that the Green's function on $V$ is independent of the choice of germ, and so we often talk about the Green's function itself without referring to the choice of germ, which is $V^+$ as a whole. \end{remark} As an easy consequence of the uniformization theorem, it turns out that $V^{\circ}$ will admit a hyperbolic structure given this setup, and so it will automatically be equipped with an equilibrium Green's function, which will be induced by standard Green's functions on $\mathbb{D}$ (see [5]). This generalizes the intuition that non-hyperbolic Riemann surfaces cannot be extended since they are already nearly as ``big" as possible. \newline \newline We will now give our definition of capacities, as mentioned in Theorem $1.1$: \begin{defin} Given the equilibrium Green's function on $V^+$ at a point $P \in V^{\circ}$ and a choice of local coordinate $z$ at $P$ so that $z(O)=0$, we define $\text{cap}_O(V):=e^{-c}$, where $c=\lim_{w \to O} g_{V,O}(w)-\log(|z(w)|^{-1})$, a limit which exists by standard results on logarithmic singularities. \end{defin} \begin{remark} A local coordinate $z$ will be fixed throughout, and our results and key quantities of interest (such as Arakelov degree) will be independent of a choice of local coordinate \end{remark} \begin{defin} Given a current $\phi$ on $V$ and a holomorphic mapping $f: V \to W$ of Riemann surfaces with finite fibers on $\text{supp}(\phi)$, we can define the pushforward current $f_*\phi$ on $W$ by $f_*\phi(y)=\int_V \phi\delta_{f^*(y)}$. \end{defin} The pushforward current can be rewritten as $f_*\phi(y)=\sum_{x \in f^{-1}(y)} e_x\phi(x)$, where $e_x$ is the ramification index at $x$. \begin{lemma} Let $f: V^+ \to W$ be a holomorphic map of Riemann surfaces that is nonconstant on $V$ and $g$ a Green's function for $V^+$. Then $f$ has finite fibers on $\text{supp}(g)$. \end{lemma} \begin{proof} Recall that the Green's function is supported on $V$, meaning that we just need to show that $f$ has finite fibers on $V$. Note that $V$ is a connected compact Riemann surface and that $f$ is holomorphic on a neighborhood of $V$. Choosing the conformal metric on the Riemann surface $V^+$ endows $V^+$ with the structure of a metric space. Now given any $x \in V$ and an infinite sequence of preimages $\{x_n, n \in \mathbb{N}\}$ of $x$, the $x_i$s have a limit point $v$ in $V$, and taking a holomorphic chart containing $v$ shows that $f$ is constant on some neighborhood of $v$, thus on all of $V$ since $V$ is connected, contradiction. \end{proof} \begin{defin} In light of Definition $2.3$ and Lemma $2.8$, we can define the pushforward $f_*g$ of a Green's function on $V$ via any nonconstant holomorphic map $f: V^+ \to W$ of Riemann surfaces, and this map will be given by $f_*g(w)=\sum_{v \in f^{-1}(w)} g(v)$. \end{defin}
\section{Bost-Charles Overflow and Arakelov Degree}
We will start by defining the Bost-Charles overflow, following Chapter 5 of [2].
\begin{defin} Let $f: V^+ \to N$ be a holomorphic map between Riemann surfaces and $O$ a point of $V^{\circ}$. Let $g:=g_{V,P}$ be the equilibrium Green's function for $V^+$ at $P$ and $\mu:=\mu_{V,P}$ the corresponding equilibrium measure. Furthermore, let $e$ the ramification index of $f$ at $O$. Then we define $$\text{Ex}(f: (V,O) \to N):=\int_V g\delta_{f^*(f(O))-eO}+\int_N f_*gf_*\mu$$ \end{defin} Given a holomorphic map with ramification index $e$, we can also consider the $e$th jet of $f$, a $\mathbb{C}$-vector space generated by the element $f^{[e]}(O)$, and equip it with a capacitary norm $||\cdot||_{V,O}^{\text{cap},\otimes e}$ which is the norm on $(T_OV^{\vee})^{\otimes e}$ induced by duality and tensor power from the capacitary norm on $T_OV$. We now recall an important result of $[2]$:
\begin{prop} (Proposition $5.4.2$ of [2]) For any nonconstant holomorphic map $f: V^+ \to \mathbb{C}$, we have that $$\text{Ex}(f: V^+ \to \mathbb{C})=\int_{(\partial V)^2} \log|f(z_1)-f(z_2)|d\mu(z_1)d\mu(z_2)-\log||f^{[e]}(O)||_{V,O}^{\text{cap},\otimes e}$$ \end{prop}
 The main goal for this section will be to produce a formula that specifically gives information about the pushforward Green's function itself.
 \begin{lemma} Suppose that $f$ is nonvanishing on $\partial V$. Then the integral $\int_{V} \log|f(x')|\mu(x')$ is finite. \end{lemma}
 \begin{proof} Since $\partial V$ is $\mathcal{C}^{\infty}$, for each $x \in \partial V$, we can find a chart $\phi_1: U_x \to \partial V$ around $x$ such that $f \circ \phi_1: U_x \to \mathbb{C}$ is $C^{\infty}$. Shrink $U_x$ slightly to some $V_x$ so that $\bar{V}_x$ is contained in $U_x$, which means that $f \circ \phi_1$ is continuous on the compact set $\bar{V}_x$. Since $f$ is nonvanishing on $\partial V$, $\log|f \circ \phi_1|$ is continuous, thus bounded on $\bar{V}_x$, hence on $V_x$. By compactness of $\partial V$, the cover of all $V_x$s can be reduced to a finite cover, and then using a partition of unity, the integral becomes a finite sum of integrals, each of which are bounded, showing that $\int_{V} \log|f(x')|\mu(x')=\int_{\partial V} \log|f(x')|\mu(x')$ is finite since $\mu$ is supported on $\partial V$, as desired. \end{proof}
We will now obtain a formula for the pushforward Green's function such that the pole is contained in the integral term. Indeed, we have that:
\begin{prop} For any holomorphic $f$ with $f(O)=0$ with $f$ nonzero on $\partial V$ and for any $x$ such that $f(x) \neq 0$, we have that $(f_*g)(f(x))=\int_{V} \log|\frac{1}{f(x)}-\frac{1}{f(x')}|\mu(x')+\int_{V} \log|f(x')|\mu(x')$ \end{prop} \begin{proof} We will need the notion of a Green's function of the diagonal associated to a $2$-form $\beta$ (see section $5.3.1$ of $[2]$). We will work with the $2$-form $\beta=0$ on $\mathbb{C}$, which has associated Green's function of the diagonal given by $g_{\mathbb{C}}: \mathbb{C}^2 \to (-\infty, \infty], (z,z') \to \log|z-z'|^{-1}$. We start by noting from $(5.4.4)$ and the definition of $f^*$ that $$\int_V f_*gf_*\delta_x=-\int_V (f, f(x))^*g_N(\mu-\delta_P)=\int_{\partial V} \log|f(x)-f(x')|\mu(x')-\log(|f(x)-f(P)|)$$ The key point is that $f_*\delta_x=\delta_{f(x)}$ and so $\int_V f_*gf_*\delta_x=(f_*g)(f(x))$. Thus we have $$(f_*g)(f(x))=\int_{\partial V} \log|f(x)-f(x')|\mu(x')-\log(|f(x)-f(O)|)$$ \newline By assumption, $f(O)=0$, and then since $f$ is nonvanishing on $\partial V$, Lemma $3.3$ implies that $\int_{V} \log|f(x')|\mu(x')$ is finite. Thus we can rewrite our relation in the following way: $$(f_*g)(f(x))=\int_{V} \log|\frac{1}{f(x)}-\frac{1}{f(x')}|\mu(x')+\int_{V} \log|f(x)f(x')|\mu(x')-\log(|f(x)|)$$ Using that $\mu$ is a probability measure on $\partial V$, we immediately find that $$(f_*g)(f(x))=\int_{V} \log|\frac{1}{f(x)}-\frac{1}{f(x')}|\mu(x')+\int_{V} \log|f(x')|\mu(x')$$ completing the proof. \end{proof} Our next goal will be to find an alternative expression for the integral $\int_{V} \log|f(x')|\mu(x')$, ultimately to relate it to an Arakelov degree. We have the following proposition, giving the precise value of the integral: \begin{prop} For any holomorphic $f$ with $f(O)=0$, nonvanishing on $\partial V$, we have that $$\int_{V} \log|f(x')|\mu(x')=\int_V g\delta_{f^*(f(O))-eO}+\log||f^{[e]}(O)||_{V,O}^{\text{cap},\otimes e}$$ \end{prop} \begin{proof} By Proposition $3.4$ and integrating with respect to the equilibrium measure, we immediately find that $$\int_V (f_*g)(f(x))\mu(x)=\int_{(\partial V)^2}  \log|f(x')|\mu(x')\mu(x)+\int_{(\partial V)^2} \log|\frac{1}{f(x)}-\frac{1}{f(x')}|\mu(x')\mu(x)$$ Noting that $\mu$ is an equilibrium measure and using the definition of the pushforward, we get that $$\int_N (f_*g)f_*\mu(x)= \int_{V} \log|f(x')|\mu(x')+\int_{(\partial V)^2} \log|\frac{1}{f(x)}-\frac{1}{f(x')}|\mu(x')\mu(x)$$ Now note that $$\int_{(\partial V)^2} \log|\frac{1}{f(x)}-\frac{1}{f(x')}|\mu(x')\mu(x)=\int_{(\partial V)^2} \log|f(x)-f(x')|\mu(x')\mu(x)-2\int_{V} \log|f(x')|\mu(x')$$ by symmetry. We thus have that $$\int_{V} \log|f(x')|\mu(x')=\int_{(\partial V)^2} \log|f(x)-f(x')|\mu(x')\mu(x)-\int_N f_*gf_*\mu$$ By Proposition $3.2$ and the definition of overflow, we have that $$\int_{(\partial V)^2} \log|f(z_1)-f(z_2)|d\mu(z_1)d\mu(z_2)-\log||f^{[e]}(P)||_{V,P}^{\text{cap},\otimes e}=\int_V g\delta_{f^*(f(P))-eP}+\int_N f_*gf_*\mu$$ and so we find that $$\int_{V} \log|f(x')|\mu(x')=\int_V g\delta_{f^*(f(P))-eP}+\log||f^{[e]}(P)||_{V,P}^{\text{cap},\otimes e}$$ as desired. \end{proof} As an immediate corollary of the work done so far, we have: \begin{corollary} Given the conditions of Proposition $3.5$, we have the following identity on $V^+ \backslash \{O\}$: $$(f_*g)(f(x))-(\int_V g\delta_{f^*(f(O))-eO}+\log||f^{[e]}(O)||_{V,O}^{\text{cap},\otimes e})=\int_{V} \log|\frac{1}{f(x)}-\frac{1}{f(x')}|\mu(x')$$ \end{corollary}
As an example of the utility of such an identity, let $V=\hat{\mathbb{C}} \backslash K^{\circ}$ for some compact domain $K$ such that $0 \in K^{\circ}$ and let $\varphi(z)=\frac{1}{z}$, which is holomorphic on $V$. Then Corollary $3.6$ recovers the classical formula for the Green's function at $\infty$ for $K$: $G_K(z, \infty)-V_{\infty}(K)=\int_K \log|z-w|\mu_{\infty}(w)$.
We will now introduce some new definitions:
\begin{defin} Let $\tilde{V}$ be a formal analytic arithmetic surface with gluing map $\psi$ and $X$ a separated scheme of finite type over $\mathbb{Z}$. Then a morphism $\tilde{V} \to X$ consists of the data $(\hat{f},f)$, where $\hat{f}: \hat{V} \to X$ is a morphism of formal schemes over $\mathbb{Z}$ and $f: V \to X(\mathbb{C})$ is a holomorphic map of the underlying Riemann surfaces, such that $\hat{f}=f \circ \psi$. \end{defin} \begin{defin} Recall that there was a canonical section $P: \text{Spec}(\mathbb{Z}) \to \hat{V}$. We define the pushforward $Q$ of $P$ as the section $\hat{f} \circ P: \text{Spec}(\mathbb{Z}) \to X$. \end{defin} Note that $Q$ is an effective Cartier divisor on $X$ and thus that the inverse image $\hat{f}^*Q$ is a Cartier divisor on $\hat{V}$. We can uniquely write $\hat{f}^*Q=eP+R$ for some divisor $R$ disjoint from $P$. \newline \newline 
Now let $f(X) \in \mathbb{Z}[[x]]$ be a power series with $f(O)=0$ and write $f(X)=a_eX^e+\sum_{n=e+1}^{\infty} a_nX^n$, so that $f$ has vanishing order $e$. We suppose that $f(X)$ is monic, meaning that $a_e=1$, and let $\varphi$ be the compositional inverse of the gluing map. As before, set $\phi=f \circ \varphi$. \begin{prop} We have that $$\log||\phi^{[e]}(O)||^{\text{cap}, \otimes e}_{V,O}=e\widehat{\text{deg}}(P^*\overline{N}_P(\tilde{V}))=e\log||d\varphi(O)||^{\text{cap}}_{V,O}$$ \end{prop}
\begin{proof} Starting with $\phi$, our regular function on the formal-analytic surface $\tilde{V}$, we get a morphism of formal-analytic arithmetic surfaces $\tilde{V} \to \mathbb{A}_{\mathbb{Z}}^1$ induced by the map $f=\hat{f}: \text{Spf}(\mathbb{Z}[[x]])=\hat{\mathcal{V}} \to \text{Spf}(\mathbb{Z}[[x]])$ and the holomorphic map $\phi: V^+ \to \mathbb{C}$, which is immediately seen to be compatible with gluings. \newline \newline Now by $7.2.17$ of $[2]$, we have the following result: $$\widehat{\text{deg}} Q^*\overline{N}_Q(X)=e\widehat{\text{deg}} P^*\overline{N}_P(\tilde{V})+\widehat{\text{deg}}(O) \cdot R-\log||\phi^{[e]}(O)||^{\text{cap}, \otimes e}_{V,O}$$ Since the morphism $\hat{f}$ is given by $f$, we know from Example $7.2.3$ of $[2]$ that the cycle $\widehat{\text{deg}}(O) \cdot R$ corresponds to the divisor of $a_e$ in $P=\text{Spec}(\mathbb{Z})$, which is just the sum of the arithmetic degrees over primes, giving $\log|a_e|=0$ since $f$ is assumed to be monic. Furthemore, since the capacitary norm on $\mathbb{A}^1(\mathbb{C})$ satisfies $||\frac{\partial}{\partial z}||=1$, we also have that $\widehat{\text{deg}} Q^*\overline{N}_Q(X)=0$. Hence we conclude that $\log||\phi^{[e]}(O)||^{\text{cap}, \otimes e}_{V,O}=e\widehat{\text{deg}}(P^*\overline{N}_P(\tilde{V}))$. \newline \newline Now we need to show that $\log||\phi^{[e]}(O)||^{\text{cap}, \otimes e}_{V,O}=e\log||d\varphi(O)||^{\text{cap}}_{V,O}$. Recall that $\phi=f \circ \varphi$ as thus writing $\varphi=\sum_{i=1}^{\infty} c_iz^i$ in a uniformizer $z$, we find that $\phi(z)=a_ec_1^ez^e+O(z^{e+1})$. However, note that $$|c_1|=\frac{||d\varphi(O)||^{\text{cap}}_{V,O}}{|z'(O)|}$$ since it corresponds to the first derivative relative to the choice of local coordinate $z$ while monicity implies that $a_e=1$, implying that $$\phi^{(e)}(z)=e!(\frac{||d\varphi(O)||^{\text{cap}}_{V,O}}{|z'(O)|})^e$$ On the other hand, we have that $\log||\phi^{[e]}(O)||^{\text{cap}, \otimes e}_{V,O}$ is given by $\log(\frac{|\phi^{(e)}(z)(z'(O))^e|}{e!})$ since we're taking the $e$th tensor power and thus extracting the coordinate-free $e$th power coordinate, and so $$\log||\phi^{[e]}(O)||^{\text{cap}, \otimes e}_{V,O}=\log(\frac{|\phi^{(e)}(z)(z'(O))^e|}{e!})=e\log||d\varphi(O)||^{\text{cap}}_{V,O}$$ as desired.  \end{proof} This result implies that if $\tilde{\mathcal{V}}$ is pseudoconvex and $f$ is monic, then $\phi=f \circ \varphi$ satisfies the first hypothesis of Conjecture $1$.
Given our Riemann surface $V^+$ and our antiholomorphic involution $\sigma$ on $V$ (which we assume extends to an antiholomorphic function on some neighborhood of $V$, thus inducing an antiholomorphic involution on $V$). The following result gives another important property of $\phi$, verifying the second hypothesis of Conjecture $1.1$:
\begin{lemma} Let $\phi$ be holomorphic on some neighborhood of $V$. Then for all $v \in V$, we have that $\phi(\sigma(v))=\overline{\phi(v)}$. \end{lemma} \begin{proof} Since $\sigma$ is an antiholomorphic involution and $O$ is a real point, we can choose some chart $z$ on the neighborhood $U'$ of $P$ compatible with the involution centered at $O$ (so that $z(O)=0$) such that $z \circ \sigma=\bar{z}$. By definition, $f$ is compatible with the real structure on $V$, which means that $\phi \circ z^{-1}$, which is a map $U \to \mathbb{C}$ for some open subset $U$ of $\mathbb{C}$ containing $0$ preserved by complex conjugation, has a power series expansion with real coefficients. This means that for any $x \in U$, we have that $\phi(z^{-1}(\overline{x}))=\overline{\phi(z^{-1}(x))}$. Now let $v=z^{-1}(x)$. We just want to show that $z^{-1}(\bar{x})=\sigma(v)$. Note that $\sigma(v)=\sigma(z^{-1}(x))$. But $\sigma \circ z^{-1}=\bar{z}^{-1}$ by definition. But if $z(w)=\bar{x}$, then $\bar{z}(w)=x$, so $\bar{z}^{-1}(x)=w=z^{-1}(\bar{x})$, as desired. Thus we have the equality $\overline{\phi(\sigma(v))}=\phi(v)$ on $U'$. \newline \newline Now we claim that $\overline{\phi(\sigma(v))}$ is holomorphic on some neighborhood of $V$. Choose any point $u$ and chart $w$ containing $u$ in its domain. By definition, we can find a chart $w'$ such that $w'\sigma w^{-1}$ is antiholomorphic on some neighborhood of $u$, and so write $w'\sigma w^{-1}=\bar{g}$. Thus we have that $\sigma w^{-1}=w'^{-1}\bar{g}$. Hence we just need to show that $\overline{\phi \circ w'^{-1} \circ \bar{g}}$ is holomorphic. However, $\phi \circ w'^{-1}$ is holomorphic, so at each point $v$, we can write $$\phi \circ w'^{-1}(y)=\sum_{i=1}^{\infty} a_iy^i$$ in some neighborhood of $v$, which means that $$\overline{\phi \circ w'^{-1}(\bar{g}(x))}=\sum_{i=1}^{\infty} \overline{a_i}g(x)^i$$ if $x$ is in this neighborhood of $u$, proving holomorphicity at $u$ and thus holomorphicity. Hence $\overline{\phi(\sigma(v))}$ and $\phi$ are holomorphic functions on some neighborhood of $V$ that agree on $U'$, which means they coincide on all of $V$, as desired. \end{proof}
Thus we see that for a monic nonconstant regular function $(f,\phi)$ on $\tilde{\mathcal{V}}$, $\phi$ satisfies both hypotheses of Conjecture $1$.
\section{Main Results}
Before proceeding to the main result, we need two more lemmas in the style of Fekete-Szëgo (which adapt a proof outline given in [6]):
\begin{lemma} Let $N$ be a positive integer and $f(x) \in \mathbb{Q}[x]$ a monic polynomial with $\text{deg}(f)=d$. Then there exist arbitrarily large positive integers $M$ such that for all $1 \leq i \leq N$, the coefficient of $x^{dM-i}$ of $f(x)^M$ is an integer. \end{lemma}
\begin{proof} Write $f(x)=\sum_{j=0}^{d} a_jx^j$ (so that $a_d=1$ and set $a_j=\frac{p_i}{q_i}$ in lowest terms (so $q_i=1$ if $a_i=0$). Let $k=\prod_{j=0}^{d-1} q_j$. For any positive integer $n$ and $0 \leq i \leq N$, the coefficient $b_{nd-i}$ of $x^{nd-i}$ of $f(x)^n$ will be a sum of terms of the form $\prod a_i$, with exactly $n$ $a_i$s in the product, such that there are at most $N$ terms in the product with $i \neq d$. This means that $k^Nb_{nd-i}$ is an integer, and thus by the multinomial theorem, it suffices to find a value of $n$ such that $k^N|\binom{n}{b_0,\cdots,b_d}$ for any nonnegative integers $b_0,\cdots,b_d$ with $\sum_{i=0}^{d} b_i=n$ and $b_d \geq n-N$ (note that $k^N$ is a fixed positive integer determined completely by $m$ and $N$). \newline \newline Let $k=\prod_{i=1}^{s} p_i^{e_i}$, its prime factorization. Since $b_0,\cdots,b_{d-1}$ are nonnegative integers with sum $\leq N$, we know that since $\binom{\sum_{i=0}^{d-1} b_i}{b_0,\cdots,b_{d-1}}$ is an integer, that $$v_p(\prod_{i=0}^{d-1} b_i!) \leq v_p((\sum_{i=0}^{d-1} b_i)!)<N$$ for each prime $p$. Now choose $n=M$ such that $v_p(M) \geq N+Nv_p(k)$ for each $p|k$. Note that $$v_p(\binom{n}{b_0,\cdots,b_d})=\sum_{i=0}^{\sum_{i=0}^{d-1} b_i} v_p(n-i)-\sum_{i=0}^{d-1} v_p(b_i!)>v_p(n)-N \geq Nv_p(k)$$ for each $p|k$, which immediately shows that $k^N|\binom{n}{b_0,\cdots,b_d}$ for all such multinomial coefficients, as desired. \end{proof}
With this, we can now prove one more lemma, which will be our ``patching" input:
\begin{lemma} Let $U \subset \mathbb{C}$ be an open bounded subset and $K=\mathbb{C} \backslash U$. Suppose that there exists some monic polynomial $m(X) \in \mathbb{R}[X]$ such that $\inf_{x \in K} |m(x)|>1$. Then there exists some monic $p(X) \in \mathbb{Z}[X]$ such that $\inf_{x \in K} |p(x)|>1$. \end{lemma}
\begin{proof} We begin by producing a polynomial $f(x) \in \mathbb{Q}[x]$ such that $\inf_{x \in K} |f(x)|>1$. Set $\epsilon=\min(\inf_{x \in K} |m(x)|-1,1)$, which is a positive real number since $\lim_{x \to \infty} m(x)=\infty$ and $\inf_{x \in C} |m(x)|>1$ on any compact subset $C$ of $K$. Write $m(x)=\sum_{i=0}^{d} a_ix^i$ so that $a_d=1$. Let $m=\sum_{i=0}^{d-1} |a_i|$. For some $r>m+2$, choose the closed disk $\overline{\mathbb{D}}_r$ containing $U$. By the density of $\mathbb{Q} \subset \mathbb{R}$, we can choose $b_i \in \mathbb{Q}, 0 \leq i \leq d-1$ such that $|a_i-b_i|<\frac{\epsilon}{2dr^{d-1}}$. Let $$f(x)=x^d+\sum_{i=0}^{d-1} b_ix^i$$ We will show that $f$ has the desired properties. First choose $x \in \overline{\mathbb{D}}_r \cap K$. Note that $$|f(x)-m(x)| \leq \sum_{i=0}^{d-1} |a_i-b_i||x|^i \leq r^{d-1}\sum_{i=0}^{d-1} |a_i-b_i|<\frac{\epsilon}{2}$$ which shows that $\inf_{x \in \overline{\mathbb{D}}_r \cap K} |f(x)|>1$. Now choose $x \not \in \overline{\mathbb{D}}_r$. Note that $$|f(x)| \geq |x|^d-(\sum_{i=0}^{d-1} |b_i||x|^i) \geq |x|^{d-1}(|x|-\sum_{i=0}^{d-1} |b_i|)$$ since $|x|>r>m+2>1$. But then $$\sum_{i=0}^{d-1} |b_i| \leq \sum_{i=0}^{d-1} |a_i|+|b_i-a_i| \leq m+\frac{\epsilon}{2r^{d-1}}<m+\frac{1}{2}$$ Thus since $r>1$, $$|f(x)| \geq |x|-m-\frac{1}{2} \geq \frac{3}{2}>1$$ Hence we conclude that $\inf_{x \in K} |f(x)|>1$. \newline \newline Square $f(x)$ if necessary so that $\text{deg}(f)>1$ and set $R=\inf_{x \in K} |f(x)|$ so that $R>1$. Choose some closed disk $\overline{\mathbb{D}}_r$ containing $U$, such that $|f(x)|>|x|$ outside $\overline{\mathbb{D}}_r$. Set $N=dk$ where $k$ is large enough so that $$R^{k-1}>r^d\frac{d}{R-1}+2, R^{k-d-1}>d\frac{1}{R-1}+2$$ By Lemma $4.1$, for fixed $N$, we can find arbitrarily large positive integers $M$ such that $f(x)^M$ has its $N$ highest coefficients integers, and so choose such an $M$, ensuring it is larger than $k$. Starting with the coefficient of $a_{Md-N}$ of $x^{Md-N}$, at step $i$, we write $Md-N+1-i=dq_i+r_i$ for some unique $q,r$ with $0 \leq r_i<d$, and then subtract $$\{a^{(i-1)}_{Md-N+1-i}\}f(x)^{q_i}x^{r_i}$$ where $a^{(i-1)}_{Md-N+1-i}$ is the coefficient of $x^{Md-N+1-i}$ after $i-1$ steps. Note that after step $i$, all coefficients for $x^{Md-N+1-i}$ and above are integers, and thus after $Md-N+1$ steps, the result will be a monic polynomial $p(X) \in \mathbb{Z}[X]$. We claim that $p$ has the desired properties. First consider $x \in \overline{\mathbb{D}}_r \backslash U$. Note that $$|p(x)| \geq |f(x)|^M-\sum_{i=1}^{Md-N+1} |f(x)|^{q_i}|x|^{r_i}>|f(x)|^M-r^d\sum_{i=1}^{Md-N+1} |f(x)|^{q_i}$$ Note that $q_i$ attains a given positive integer value exactly $d$ times, starting at $M-k$, and $0$ once, so we get $$\sum_{i=1}^{Md-N+1} |f(x)|^{q_i} \leq d\frac{|f(x)|^{M-k+1}-1}{|f(x)|-1}< d\frac{|f(x)|^{M-k+1}}{|f(x)|-1}$$ and thus $$|p(x)|>|f(x)|^{M-k+1}(|f(x)|^{k-1}-dr^d\frac{1}{|f(x)|-1}) \geq R^{k-1}-dr^d\frac{1}{R-1}>2$$ Thus $\inf_{x \in \overline{\mathbb{D}}_r \cap K} |p(x)|>1$. \newline \newline Now suppose that $x \not \in \overline{\mathbb{D}}_r$. Identical reasoning now yields $$|p(x)|>|f(x)|^{M-k+1}(|f(x)|^{k-1}-d|x|^d\frac{1}{|f(x)|-1})$$ By definition, we have that $|f(x)|>|x|$ and so $$|p(x)|>|f(x)|^{M-k+1+d}(|f(x)|^{k-d-1}-d\frac{1}{|f(x)|-1}) \geq R^{k-d-1}-d\frac{1}{R-1}>2$$ showing that $\inf_{x \in K} |p(x)|>1$, as desired. \end{proof}
We will now prove an important proposition which will quickly imply Theorem $1.1$.
 We will now proceed to the proof of Theorem $1.1$:
\begin{proof} Assuming Conjecture $1$, Lemma $4.3$ implies that we can find a monic polynomial $p(X) \in \mathbb{Z}[X]$ such that $$\inf_{x \in V} |p(\frac{1}{\phi(x)})|>1$$ which then immediately implies that there is an open neighborhood $U$ of $V$ such that $\inf_{x \in U} |p(\frac{1}{\phi(x)})|>1$. Now let $q(X)=p(1/X)$. We then have that $\inf_{x \in U} |q(\phi(x)|>1$. Then note that $q(\phi(z))=p(\frac{1}{\phi(z)})$, which is meromorphic on $U$ since $p$ is a polynomial and $\phi$ is holomorphic on $U$, and so $\frac{1}{q(\phi(z))}$ is also meromorphic on $U$. Furthermore, it follows immediately from the inequality $\inf_{x \in U} |q(\phi(x)|>1$ that $\frac{1}{q(\phi(z))}$ is bounded on $U$ and so any singularities are removable, which implies that $\frac{1}{q(\phi(z))}$ is in fact holomorphic on $U$. \newline \newline Thus for any bounded sequence $\{a_n, n \in \mathbb{N}\}$ of integers, $$\sum_{n=1}^{\infty} \frac{a_n}{q(\phi(z))^n}$$ is holomorphic on $U$ since it is a uniform limit of holomorphic functions on $U$. Furthermore, as power series, note that $\frac{1}{q(X)}=\frac{X^m}{X^mp(\frac{1}{X})}$, and since $p$ is monic, $X^mp(\frac{1}{X})$ has constant term $1$, which implies that the formal series $\frac{1}{q(X)^n} \in X^{mn}\mathbb{Z}[[x]]$. In particular, for any sum $\sum_{n=1}^{\infty} \frac{a_n}{q(X)^n}$, we get a convergent formal series $g \in \mathbb{Z}[[x]]$, and so these formal series provide our continuum of holomorphic power series on $U$ such that $g(\phi(x))=(g \circ f)(\varphi(x))$ is holomorphic on $U$ by composition.\end{proof}
\section*{Acknowledgments}
The author would like to thank his mentor Professor Vesselin Dimitrov for introducing him to this interesting
problem and offering suggestions. This project would not have been possible without his guidance.
The author also wishes to thank Shirley and Carl Larson for their generous support through an eponymous SURF fellowship.

\end{document}